\documentclass[11pt]{amsart}

\usepackage{amsmath, amsthm, amssymb, amsfonts, enumerate}
\usepackage[colorlinks=true,linkcolor=blue,urlcolor=blue]{hyperref}
\usepackage{dsfont}
\usepackage{color}
\usepackage{geometry}
\usepackage{todonotes}

\usepackage{graphicx}
\usepackage{caption}

\usepackage{float}

\usepackage{mathtools}

\def \cC{\mathcal{C}}
\def \cD{\mathcal{D}}

\def \cF{\mathcal{F}}

\def \cL{\mathcal{L}}

\def \cQ{\mathcal{Q}}

\def \P{\mathsf P}

\def \PP{\widetilde{\mathsf{P}}}
\def \E{\mathsf E}
\def \EE{\widetilde{\mathsf{E}}}
\def \N{\mathbb{N}}
\def \R{\mathbb{R}}

\def \ud{\mathrm{d}}
\def \e{\mathrm{e}}

\newcommand{\eps}{\varepsilon}
\newcommand{\ind}{\mathbf{1}}

\newtheorem{theorem}{Theorem}[section]
\newtheorem{lemma}[theorem]{Lemma}
\newtheorem{corollary}[theorem]{Corollary}
\newtheorem{proposition}[theorem]{Proposition}
\newtheorem{remark}[theorem]{Remark}

\title[Optimal stopping for the exponential of a Brownian bridge]{Optimal stopping \\ for the exponential of a Brownian bridge}
\author[De Angelis and Milazzo]{Tiziano De Angelis \and Alessandro Milazzo}
\address{T.~De Angelis: School of Mathematics, University of Leeds, Woodhouse Lane, LS2 9JT Leeds, UK.}
\email{\href{mailto:t.deangelis@leeds.ac.uk}{t.deangelis@leeds.ac.uk}}
\address{A.~Milazzo: Department of Mathematics, Imperial College London, 16-18 Princess Gardens, SW7 1NE London, UK.}
\email{\href{mailto:a.milazzo16@imperial.ac.uk}{a.milazzo16@imperial.ac.uk}}
\date{\today}

\numberwithin{equation}{section}

\begin{document}

\begin{abstract}
In this paper we study the problem of stopping a Brownian bridge $X$ in order to maximise the expected value of an exponential gain function. In particular, we solve the stopping problem
\[
\sup_{0\le \tau\le 1}\E[\e^{X_\tau}],
\]
which was posed by Ernst and Shepp in their paper [Commun.~Stoch.~Anal., 9 (3), 2015, pp.~419--423] and was motivated by bond selling with non-negative prices.

Due to the non-linear structure of the exponential gain, we cannot rely on methods used in the literature to find closed-form solutions to other problems involving the Brownian bridge. Instead, we must deal directly with a stopping problem for a time-inhomogeneous diffusion. We develop techniques based on pathwise properties of the Brownian bridge and martingale methods of optimal stopping theory, which allow us to find the optimal stopping rule and to show regularity of the value function.
\end{abstract}

\maketitle

\section{Introduction}

Problems of optimal stopping involving Brownian bridge have a long history, dating back to the early days of modern optimal stopping theory. The first results were obtained by Dvoretzky \cite{D67} and Shepp \cite{shepp1969}. Both authors considered stopping of a Brownian bridge to maximise its expected value. Dvoretzky proved existence of an optimal stopping time and Shepp  provided an explicit solution in terms of the first time the Brownian bridge (pinned at zero at time $T=1$) exceeds a boundary of the form $t\mapsto a\sqrt{1-t}$, for $t\in[0,1]$ and a suitable $a>0$.

Few years later, F\"ollmer \cite{follmer1972} extended the study to the case of a Brownian bridge whose pinning point is random with normal distribution. He showed that the optimal stopping time is the first time the process crosses a time-dependent boundary and the stopping set may lie either above or below the boundary, depending on the variance of the pinning point's distribution.

More recently, Ekstr\"om and Wanntorp \cite{ekstrom2009optimal} studied optimal stopping of a Brownian bridge via the solution of associated free boundary problems. They recovered results by Shepp and extended the analysis by finding explicit solutions to some examples with more general gain functions than the linear case.

Optimal stopping of a Brownian bridge with random pinning point or random pinning time were also studied in \cite{ekstrom2017} and \cite{glover2019}, respectively. In \cite{ekstrom2017}, the authors consider more general versions of the problem addressed in \cite{follmer1972} and, among other things, they give general sufficient conditions for optimal stopping rules in the form of a hitting time to a one-sided stopping region. In \cite{glover2019}, the author provides sufficient conditions for a one-sided stopping set and is able to solve the problem in closed form for some choices of the pinning time's distribution.

Problems of optimal stopping for Brownian bridge have attracted significant attention from the mathematical finance community thanks to their application to trading. Already in 1970, Boyce \cite{boyce1970} proposed applications of Shepp's results to bond trading. In that context the pinning effect of the Brownian bridge captures the well-known {\em pull-to-par} mechanism of bonds. Many other applications to finance have appeared in recent years, motivated by phenomena of stock pinning (see, e.g.~\cite{avellaneda2003} and \cite{jeannin2008modeling} among many others). Explicit results for some problems of optimal double stopping of a Brownian bridge, also inspired by finance, were obtained in \cite{bardoux2015}.

In our paper we study a problem that was posed by Ernst and Shepp in Section 3 of \cite{ernst2016}. In particular, we are interested in finding the optimal stopping rule that maximises the expected value of the exponential of a Brownian bridge which is constrained to be equal to zero at time $T=1$. Besides the pure mathematical interest, this problem is better suited to model bond/stock trading situations than its predecessors with linear gain function. Indeed, the exponential structure avoids the unpleasant feature of negative asset prices, whilst retaining the pinning effect discussed above. Questions concerning stopping the exponential of a Brownian bridge were also considered in \cite{leung2018} in a model inspired by financial applications. In fact, in \cite{leung2018} authors consider a more general model than ours and allow a random pinning point. However, the complexity of the model is such that the analysis is carried out mostly from a numerical point of view.

In this work we prove that the optimal stopping time for our problem is the first time the Brownian bridge exceeds a time-dependent optimal boundary $t\mapsto b(t)$, which is non-negative, continuous and non-increasing on $[0,1]$. The boundary can be computed numerically as the unique solution to a suitable integral equation of Volterra type (see Section \ref{sec:II}). The full analysis that we perform relies on four equivalent formulations of the problem (see \eqref{MarkovianPb}, \eqref{MarkovianPb1}, \eqref{MarkovianPb2} and \eqref{MarkovianPbTilde}), which are of interest in their own right, and offer different points of view on the problem.

Our study reveals interesting features of the value function $v$. Indeed, we can prove that $v$ is continuously differentiable on $[0,1)\times\R$, both with respect to time and space, with second order spatial derivative which is continuous up to the optimal boundary (notice that this regularity goes beyond the standard smooth-fit condition in optimal stopping). Notice, however, that the value function {\em is not} continuous at $\{1\}\times(-\infty,0)$, due to the pinning behaviour of the Brownian bridge as $t\to 1$. 

We extend the existing literature in several directions. The exponential structure of the gain function makes it impossible to use scaling properties that are central in all the papers where explicit solutions were obtained (see, e.g., \cite{shepp1969}, \cite{ekstrom2009optimal}, \cite{bardoux2015}, \cite{glover2019}). For this reason we must deal directly with a stopping problem for a time-inhomogeneous diffusion. Optimal boundaries for such problems are hard to come by in the literature and, in order to prove monotonicity of the boundary (which is the key to the subsequent analysis), we have developed a method based on pathwise properties of the Brownian bridge and martingale theory (see Theorem \ref{DecreasingBarrier}). The task is challenging because there is no obvious comparison principle for sample paths of Brownian bridges $X^{t,x}$ and $X^{t',x}$ starting from a point $x\in\R$ at different instants of time $t\neq t'$. Hence, our approach could be used in other optimal stopping problems involving time-inhomogeneous diffusions. 

It is worth noticing that, in Section 5 of \cite{ekstrom2017}, authors also obtain a characterisation of the optimal boundary via integral equations. However, in that case a time-change of the Brownian bridge and linearity of the gain function are used to infer monotonicity of the boundary. 

The paper is organised as follows. In Section \ref{sec:PF} we provide some background notions on the Brownian bridge and formulate the stopping problem. In Section \ref{sec:value} we prove continuity of the value function and existence of an optimal boundary. In Section \ref{sec:freeb} we prove that the boundary is monotonic non-increasing, continuous and bounded on $[0,1]$ and find its limit at time $t=1$. In Section \ref{sec:C1} we find $C^1$ regularity for the value function and we derive the integral equation that uniquely characterises the optimal boundary. In Section \ref{sec:numerics} we solve the integral equation numerically using a Picard's iteration scheme and we provide plots of the optimal boundary and of the value function. We also illustrate numerically the convergence of the algorithm for the boundary and the dependence of the boundary on the pinning time of the Brownian bridge.

\section{Problem formulation}\label{sec:PF}
We consider a complete filtered probability space $(\Omega, \cF, (\cF_t)_{t\ge 0}, \P)$, equipped with a standard Brownian motion $W:=(W_t)_{t\ge 0}$. With no loss of generality, we assume that $(\cF_t)_{t\ge 0}$ is the filtration generated by $W$ and augmented with $\P$-null sets. Further, we denote by $X:=(X_t)_{t\in[0,1]}$ a Brownian bridge pinned at zero at time $T=1$, i.e.~such that $X_1=0$. If the Brownian bridge starts at time $t\in[0,1)$ from a point $x\in\R$, we sometimes denote it by $(X^{t,x}_{s})_{s\in[t,1]}$ in order to keep track of the initial condition. 

It is well-known that, given an initial condition $X_t=x$ at time $t\in[0,1)$, the dynamics of $X$ can be described by the following stochastic differential equation (SDE):
\begin{align}\label{MarkovX}
\ud X_s=-\frac{X_s}{1-s}\ud s+\ud W_s, \qquad s\in[t,1).
\end{align}
The unique strong solution of the SDE \eqref{MarkovX} is given by
\begin{align}\label{MarkBrownBridge}
X^{t,x}_s &=(1-s)\left(\frac{x}{1-t}+\int_t^s \frac{\mathrm{d} W_u}{1-u} \right), \quad s\in[t,1]. 
\end{align}

The expression in \eqref{MarkBrownBridge} allows to identify (in law) the process $X^{t,x}$ with the process $Z^{t,x}:=(Z^{t,x}_s)_{s\in[t,1]}$ given by
\begin{equation}\label{MarkBrownBridge1}
Z^{t,x}_s:=\frac{1-s}{1-t}x+\sqrt{\frac{1-s}{1-t}}W_{s-t}, \quad s\in[t,1].
\end{equation}
That is, we have 
\begin{align}\label{eq:laws}
\mathsf{Law}\big(X^{t,x}_s\,,\, s\in[t,1]\big)=\mathsf{Law}\big(Z^{t,x}_s\,,\,s\in[t,1]\big)
\end{align}
for any initial condition $(t,x)\in[0,1]\times\R$. In the rest of the paper we will often use the notations $\E_{t,x}[\,\cdot\,]=\E[\,\cdot\,|X_t=x]$ and, equivalently, $\E_{t,x}[\,\cdot\,]=\E[\,\cdot\,|Z_t=x]$.

Using the above mentioned identity in law of $X$ and $Z$, along with well-known distributional properties of the Brownian motion, it can be easily checked that 
\begin{align}\label{eq:ui}
\E_{t,x}\left[\sup_{t\leq s\leq 1}\e^{X_s}\right]\le \e^{|x|}\E\left[\e^{S_1}\right]<+\infty,
\end{align}
where $S_1:=\sup_{0\le s\le 1}|W_s|$. The random variable $S_1$ will be used several times in what follows and we denote 
\begin{align}\label{eq:c1}
c_1:=\E\Big[\e^{S_1}\Big]\quad\text{and}\quad c_2:=\E\Big[S_1\e^{S_1}\Big].
\end{align}
\vspace{+4pt}

\subsection{The stopping problem}
Our objective is to study the optimal stopping problem
\begin{equation}\label{MarkovianPb}
v(t,x)=\sup_{0\leq \tau\leq 1-t} \E_{t,x}\Big[\e^{X_{t+\tau}}\Big], \quad \text{for $(t,x)\in[0,1]\times\R$},
\end{equation}
where $\tau$ is a random time such that $t+\tau$ is a $(\cF_s)_{s\geq t}$-stopping time (in what follows we simply say that $\tau$ is a $(\cF_s)_{s\geq t}$-stopping time, as no confusion shall arise). 
Thanks to \eqref{eq:ui}, we can rely upon standard optimal stopping theory to give some initial results. In particular, we split the state space $[0,1]\times\R$ in a continuation region $\cC$ and a stopping region $\cD$, respectively given by 
\begin{align}
&\cC:=\{(t,x)\in[0,1]\times\R : \: v(t,x)>\e^x\},\\[+4pt]
&\cD:=\{(t,x)\in[0,1]\times\R : \: v(t,x)=\e^x\}.
\end{align}

Then, for any $(t,x)\in[0,1]\times\R$, the smallest optimal stopping time for problem \eqref{MarkovianPb} is given by (see, e.g., \cite[Thm.~D.12, Appendix D]{karatzas1998methods})
\begin{equation}\label{Tau*}
\tau^*:=\inf\{s\in [0,1-t]: (t+s,X_{t+s})\in \cD\},\quad\text{$\P_{t,x}$-a.s.}
\end{equation}
We will sometimes use the notation $\tau^*_{t,x}$ to keep track of the initial condition of the time-space process $(t,X)$.

Moreover, standard theory on the Snell envelope also guarantees (see, e.g., \cite[Thm.~D.9, Appendix D]{karatzas1998methods}) that the process $V:=(V_t)_{t\in[0,1]}$ defined by $V_t:=v(t,X_t)$ is a right-continuous, non-negative, $\P$-super-martingale and that $V^*:=(V_{t\wedge\tau^*})_{t\in[0,1]}$ is a right-continuous, non-negative, $\P$-martingale.

To conclude this section, we show two further formulations of problem \eqref{MarkovianPb} that will become useful in our analysis. The former uses \eqref{eq:laws} and the fact that, thanks to the above discussion, we only need to look for optimal stopping times in the class of entry times to measurable sets. Hence, we have
\begin{equation}\label{MarkovianPb1}
v(t,x)=\sup_{0\leq \tau\leq 1-t} \E_{t,x}\Big[\e^{Z_{t+\tau}}\Big],\quad\text{for $(t,x)\in[0,1]\times\R$}.
\end{equation}

The second formulation, instead, uses ideas originally contained in \cite{jaillet1990variational}. In particular, for any fixed $t\in[0,1]$ and any $(\cF_s)_{s\ge t}$-stopping time $\tau\in[0,1-t]$, we can define an $(\hat{\cF}_s)_{0\leq s\leq 1}$ stopping time $\theta\in[0,1]$ such that $\tau=\theta(1-t)$ and $\hat{\cF}_s=\cF_{s(1-t)}$. In addition to this, notice that
\begin{align}\label{eq:laws2}
\mathsf{Law}\big(W_{s(1-t)}\,,\, s\ge0\big)=\mathsf{Law}\big(\sqrt{1-t} \, W_s\,,\,s\ge0 \big).
\end{align}
Therefore, problem \eqref{MarkovianPb1} (hence problem \eqref{MarkovianPb}) can be rewritten as
\begin{equation}\label{MarkovianPb2}
v(t,x)=\sup_{0\leq \theta\leq 1} \E\left[\exp\left((1-\theta)x+\sqrt{(1-\theta)(1-t)}W_\theta\right)\right].
\end{equation}
This last formulation of the problem has the advantage that the domain of admissible stopping times $\theta$ is independent of the initial time $t$.

\begin{remark}\label{rem:T}
There is no loss of generality in our choice of a pinning time $T=1$ and a pinning point $\alpha=0$. We could equivalently choose a generic pinning time $T>t\geq 0$ and a generic pinning point $\alpha\in\R$ and consider the dynamics
\begin{align}
\ud X_s=-\frac{X_s-\alpha}{T-s}\ud s+\ud W_s, \qquad s\in[t,T).
\end{align}
Then, the analysis in the next sections would remain valid up to obvious tweaks.
\end{remark}

\section{Continuity of the value function and existence of a boundary}\label{sec:value}

In this section we prove some properties of the value function, including its continuity, and derive the existence of a unique optimal stopping boundary. It follows immediately from \eqref{eq:ui} that 
the value function is non-negative and uniformly bounded on compact sets. In particular, we have
	\begin{align}\label{eq:vb}
	0\leq v(t,x)\leq c_1 \e^{|x|}, \quad \text{for all $(t,x)\in [0,1]\times \R$},
	\end{align}
	where $c_1>0$ is given by \eqref{eq:c1}.

\begin{proposition}\label{xVcontinuity}
	The map $x\mapsto v(t,x)$ is convex and non-decreasing. Moreover, for any compact set $K\subset \R$ there exists $L_K>0$ such that
	\[
	\sup_{t\in[0,1]}|v(t,y)-v(t,x)|\leq L_K|y-x|, \quad \text{for all $x,y\in K$.}
	\]
	\end{proposition}

\begin{proof}
Convexity of $x\mapsto v(t,x)$ follows from linearity of $x\mapsto Z^{t,x}_s$ (see \eqref{MarkBrownBridge1}), convexity of the map $x\mapsto \e^x$ and the well-known inequality $\sup(a+b)\le \sup a+\sup b$.

Monotonicity can be easily deduced by, e.g., the explicit dependence of \eqref{MarkovianPb2} on $x\in\R$. As for the Lipschitz continuity, the claim is trivial for $t=1$ since $v(1,x)=e^x$. For the remaining cases, fix $t\in[0,1)$  and let us fix $y\geq x$. Denote $\tau_y:=\tau^*_{t,y}$, then by monotonicity of $v(t,\,\cdot\,)$, the fact that $\tau_y$ is sub-optimal for $v(t,x)$ and simple estimates, we obtain
	\begin{align*}
	0 & \leq v(t,y)-v(t,x) \\
	&\leq \E\Big[\e^{Z^{t,y}_{t+\tau_y}}-\e^{Z^{t,x}_{t+\tau_y}}\Big] \\
	& = \E\left[\left(\exp\left(\frac{1-(t+\tau_y)}{1-t}y\right)-\exp\left(\frac{1-(t+\tau_y)}{1-t}x\right)\right)\exp\left(\sqrt{\frac{1-(t+\tau_y)}{1-t}}W_{\tau_y}\right)\right]\\
	&\le \E\left[\left(\frac{1-(t+\tau_y)}{1-t}\right)\exp\left(\sqrt{\frac{1-(t+\tau_y)}{1-t}}W_{\tau_y}\right)\right]\e^{|x|\vee |y|}(y-x)\\
	&\le \E\Big[\e^{S_1}\Big]\e^{|x|\vee |y|}(y-x).
	\end{align*}
Hence, the claim follows with $L_K:= c_1 \max_{x\in K}e^{|x|}$.
\end{proof}

Next we show that the value function is locally Lipschitz in time on $[0,1)\times\R$. However, it fails to be continuous at $\{1\}\times(-\infty,0)$.
\begin{proposition}\label{prop:cont-t}
For any $T<1$ and any $0\le t_1<t_2\le T$, we have
\begin{align}\label{eq:lipt}
|v(t_2,x)-v(t_1,x)|\leq \frac{c_2e^{|x|}}{2\sqrt{1-T}}(t_2-t_1), \quad \text{for $x\in\R$},
\end{align}
with $c_2>0$ as in \eqref{eq:c1}. Moreover, 
\begin{align}
\label{eq:discont1}&\lim_{t\to 1}v(t,x)=e^x,\qquad\qquad\text{for $x\ge 0$},\\
\label{eq:discont2}&\liminf_{t\to 1}v(t,x)\ge 1> e^x,\qquad\text{for $x< 0$}.
\end{align}
\end{proposition}

\begin{proof}	
For the proof of \eqref{eq:lipt} we will refer to the problem formulation in \eqref{MarkovianPb2}. Fix $0\leq t_1<t_2\leq T<1$ and let $\theta_2:=\theta^*_{t_2,x}$ be the optimal stopping time for $v(t_2,x)$. Then, given that $\theta_2$ is admissible and sub-optimal for the problem with value $v(t_1,x)$, we have
	\begin{align}\label{eq:lip1}
	&v(t_2,x)-v(t_1,x) \\
	&\leq \E\Big[\e^{(1-\theta_2)x} \Big(\e^{\sqrt{(1-\theta_2)(1-t_2)}W_{\theta_2}}-\e^{\sqrt{(1-\theta_2)(1-t_1)}W_{\theta_2}}\Big)\Big]\notag\\
	&\leq \e^{|x|}\E\Big[\e^{\sqrt{(1-\theta_2)(1-t_1)}|W_{\theta_2}|}\sqrt{(1-\theta_2)}|W_{\theta_2}|\Big]\left(\sqrt{1-t_1}-\sqrt{1-t_2}\right)\notag\\
	&\leq \e^{|x|}\E\Big[S_1\e^{S_1}\Big]\frac{t_2-t_1}{2\sqrt{1-T}}.\notag
	\end{align}
	Now, setting $\theta_1:=\theta^*_{t_1,x}$ we notice that $\theta_1$ is admissible and sub-optimal for the problem with value $v(t_2,x)$. Then, arguments as above give
\[
v(t_2,x)-v(t_1,x)\geq -\e^{|x|}\E\Big[S_1\e^{S_1}\Big]\frac{t_2-t_1}{2\sqrt{1-T}},
\]
	which, combined with \eqref{eq:lip1}, implies \eqref{eq:lipt}.
	
	Finally, we show \eqref{eq:discont1} and \eqref{eq:discont2}. Notice first that $v(1,x)=e^x$ and $v(t,x)\geq \e^x$ for $t\in[0,1)$. Pick $x\geq 0$, then by \eqref{MarkovianPb1} we have $\e^x\leq v(t,x)\leq \e^x\E\left[e^{S_{1-t}}\right]$ which implies \eqref{eq:discont1} by dominated convergence and using that $S_{1-t}\to0$ as $t\to 1$. If $x<0$, instead, the sub-optimal strategy $\tau=1-t$ gives $v(t,x)\geq 1$. Hence, $\liminf_{t\to 1} v(t,x)\geq 1> \e^x=v(1,x)$ as in \eqref{eq:discont2}.	
\end{proof}

As a corollary of the two propositions just stated, we have that $\cC$ is an open set. Combining this fact with the martingale property (in $\cC$) of the value function, we obtain that $v\in C^{1,2}(\cC)$ and it solves the free boundary problem (see, e.g., arguments as in the proof of Theorem 7.7 in Chapter 2 Section 7 of \cite{karatzas1998methods})
\begin{align}
\label{freeb1}\left(\partial_t +\tfrac{1}{2}\partial_{xx}-\tfrac{x}{1-t}\partial_x\right) v(t,x) &= 0,\qquad\: (t,x)\in\cC\\
\label{freeb2}v(t,x) &= \e^x,\qquad(t,x)\in\partial\cC,
\end{align} 
where $\partial_t$, $\partial_x$ and $\partial_{xx}$ denote the time derivative, the first spatial derivative and the second spatial derivative, respectively.

For future reference, we also denote by $\cL$ the second order differential operator associated with $X$. That is
\begin{align}\label{eq:L}
(\cL f)(t,x):=\tfrac{1}{2}\partial_{xx}f(t,x)-\tfrac{x}{1-t}\partial_xf(t,x),\qquad\text{for any $f\in C^{0,2}(\R^2)$}.
\end{align}

\subsection{Existence of an optimal boundary}
In order to prove the existence of an optimal boundary it is convenient to perform a change of measure in our problem formulation \eqref{MarkovianPb}. In particular, using the integral form of \eqref{MarkovX} (upon setting $B_\tau:= W_{t+\tau}-W_t$), we have
\begin{align*}
\E\left[\exp(X^{t,x}_{t+\tau})\right] &= \E\left[\exp\left(x+B_\tau-\int_0^\tau\frac{X^{t,x}_{t+s}}{1-(t+s)}\, \mathrm{d} s\right)\right] \\
&= \e^x\E\left[\exp\big(B_\tau-\tfrac 1 2 \tau\big)\exp\left(\int_0^\tau \left(\frac 1 2 -\frac{X^{t,x}_{t+s}}{1-(t+s)}\right)\mathrm{d} s\right)\right]\\
&=\e^x\EE\left[\exp\left(\int_0^\tau \left(\frac 1 2 -\frac{X^{t,x}_{t+s}}{1-(t+s)}\right)\mathrm{d} s\right)\right],
\end{align*}
where
\[
\frac{\ud\PP}{\ud\P}\Big|_{\cF_{1}}:=\exp\left(B_{1-t}-\frac {1} {2} (1-t)\right),%\qquad t\in[0,1],
\]
defines a new equivalent probability measure $\PP$ on $(\Omega,\cF)$ and the associated expected value $\EE$. Under $\PP$, we have
\begin{align}\label{MarkovXTilde}
\ud X^{t,x}_s=\left(1-\frac{X^{t,x}_s}{1-s}\right)\mathrm{d} s+\mathrm{d} \widetilde{W}_s, \qquad\text{for $s\in[t,1]$},
\end{align}
with $X^{t,x}_t=x$, and with $\widetilde W_t:=W_t-t$ defining a $\PP$-Brownian motion by Girsanov's theorem.

Thanks to this transformation of the expected payoff, it is clear that solving problem \eqref{MarkovianPb} is equivalent to solving
\begin{equation}\label{MarkovianPbTilde}
\tilde{v}(t,x):=\sup_{0\leq \tau \leq 1-t}\EE\left[\exp\left(\int_0^\tau \left(\frac 1 2 -\frac{X^{t,x}_{t+s}}{1-(t+s)}\right)\mathrm{d} s\right)\right].
\end{equation}
Notice that, indeed, $v(t,x)=\e^x\tilde{v}(t,x)$ implies that 
\[
\cC=\{(t,x)\in[0,1]\times\R : \: \tilde{v}(t,x)>1\}.
\]
Moreover, since $V$ is a $\P$-super-martingale and $V^*$ is a $\P$-martingale then, as a consequence of Girsanov theorem, the process $\widetilde{V}:=(\widetilde{V}_t)_{t\in[0,1]}$ defined as
\begin{align}\label{eq:vtilde}
\widetilde V_t:=\exp\left(\int_0^t \left(\frac 1 2 -\frac{X_{s}}{1-s}\right)\mathrm{d} s\right)\tilde v(t,X_t)
\end{align}
is a $\PP$-super-martingale and $\widetilde V^*:=(\widetilde{V}_{t\wedge\tau^*})_{t\in[0,1]}$ is a $\PP$-martingale, with $\tau^*$ as in \eqref{Tau*}.

Using this formulation, we can easily obtain the next result.
\begin{proposition}
There exists a function $b:[0,1]\to\R_+$ such that 
\begin{align}\label{eq:bC}
\cC=\{(t,x)\in[0,1)\times\R\,:\,x<b(t)\}.
\end{align}
\end{proposition}
\begin{proof}
Thanks to the pathwise uniqueness of the Brownian bridge, it is clear that for any $x\leq x'$ we have, $\P$-a.s.~(hence also $\PP$-a.s.) 
\[
X_s^{t,x}\leq X_s^{t,x'}, \quad \text{for all $s\in[t,1]$}.
\]
Using such comparison principle and \eqref{MarkovianPbTilde}, it is easy to show that $x\mapsto \tilde{v}(t,x)$ is non-increasing. This means, in particular, that if $(t,x)\in\cD$, then $(t,x')\in\cD$ for all $x'\ge x$. Then, setting $b(1) := 0$, we define 
\begin{align}\label{Boundary}
 b(t) :=&\sup\{x\in\R: \: \tilde{v}(t,x)>1\}\\
=&\sup\{x\in\R: \: v(t,x)>\e^x\}, \quad t\in[0,1),\notag
\end{align}
and \eqref{eq:bC} holds by continuity of the value function. For future reference, notice that \eqref{Boundary} and \eqref{eq:discont1}--\eqref{eq:discont2} give also $b(1)=0$.

It remains to show that $b(t)\ge 0$ for all $t\in[0,1]$. By choosing the stopping rule $\tau=1-t$, one has $v(t,x)\geq 1>\e^x$ for $x<0$ and any $t\in[0,1)$. Hence, 
\[
[0,1)\times(-\infty,0)\subset\cC,
\] 
and the claim follows.
\end{proof}
As a straightforward consequence of the proposition above and \eqref{Tau*}, we have
\begin{equation}\label{eq:tau*}
\tau^*_{t,x}=\inf\{s\in[0,1-t]\,:\, X^{t,x}_{t+s}\geq b(t+s)\}.
\end{equation}

\section{Regularity of the optimal boundary}\label{sec:freeb}

In this section we show that the optimal boundary is monotonic, continuous and bounded. We will then use these properties to derive smoothness of the value function across the optimal boundary, in the next section. 

By an application of Dynkin's formula we know that, given any initial condition $(t,x)\in[0,1)\times\R$, any stopping time $\tau\in[0,1-t]$ and a small $\delta>0$ we have
\begin{align}
v(t,x)\ge\E_{t,x}\Big[\e^{X_{t+\tau\wedge \delta}}\Big]=\e^x+\E_{t,x}\left[\int_0^{\tau\wedge \delta}\e^{X_{t+s}}\left(\frac{1}{2}-\frac{X_{t+s}}{1-(t+s)}\right)\ud s\right].
\end{align}
This shows that immediate stopping can never be optimal inside the set
\begin{align}\label{eq:Q}
\cQ:=\left\{(t,x)\in[0,1)\times\R: x<\tfrac 1 2 (1-t)\right\},
\end{align}
and so $\cQ\subseteq\cC$.

The next result, concerning monotonicity of the optimal boundary, is crucial for the subsequent analysis of the stopping set and of the value function. Monotonicity of optimal boundaries is relatively easy to establish in optimal stopping problems when the underlying diffusion is time-homogeneous and the gain function is independent of time. In our case, the latter condition holds but our diffusion is time-dependent, hence new ideas are needed in the proof of the theorem below. We also remark that, while in some stopping problems of a Brownian bridge (see, e.g., \cite{ekstrom2009optimal}) it is possible to rely upon a time-change in order to formulate an auxiliary equivalent stopping problem for a time-homogeneous diffusion (see \cite{pederson2000solving}), this is not the case here, due to the exponential nature of the gain function.
\begin{theorem}\label{DecreasingBarrier}
	The optimal boundary $t\mapsto b(t)$ is non-increasing on $[0,1]$. 
\end{theorem}
\begin{proof}
It is sufficient to show that, for any fixed $x\in\R$, the map $t\mapsto v(t,x)$ is non-increasing on $[0,1)$. Indeed the latter implies monotonicity of the boundary on $[0,1]$ by definition \eqref{Boundary} and using that $b(t)\geq 0$ for all $t\in[0,1)$ and $b(1)=0$.
	
Recalling \eqref{freeb1} and using convexity of $x\mapsto v(t,x)$, we obtain 	
\begin{equation}\label{BoundVt}
\partial_t v(t,x)\leq \frac{x}{1-t}\partial_x v(t,x), \quad \text{for all } (t,x)\in \cC,
\end{equation}
and, in particular, 
\begin{align}\label{eq:dtv0}
\partial_t v(t,x)\leq 0,\quad\text{for all $(t,x)\in[0,1)\times(-\infty,0]$,}
\end{align}
thanks to the fact that $\cQ\subseteq \cC$ (see \eqref{eq:Q}) and $\partial_x v\ge 0$ in $\cC$ (Proposition \ref{xVcontinuity}).
	
Notice that if $(t,x)\in \cD\setminus\partial\cC$ then $v(t,x)=\e^x$ and $\partial_t v(t,x)=0$. Since $t\mapsto v(t,x)$ is continuous on $[0,1)$, it only remains to prove that $\partial_t v(t,x)\leq 0$ for $(t,x)\in \cC$ with $x>0$. For that we proceed in two steps.
\vspace{+4pt}

{\em Step 1}. (Property of $t\mapsto X^{t,x}$). Consider $(t,x)\in \cC$ with $x>0$ and $0<\eps\leq t<1$, for some $\eps>0$. For $s\in[0,1-t]$ we denote
\begin{align}\label{eq:Y0}
Y^{t,x;\eps}_{t+s}:= X^{t,x}_{t+s}-X^{t-\eps,x}_{t-\eps+s}.
\end{align}
Since $(t,x)$ is fixed, we simplify the notation and set $Y^\eps_{t+s}:=Y^{t,x;\eps}_{t+s}$, for $s\in[0,1-t]$. Next, for some small $\delta>0$, we let $t_\delta:=(1-t-\delta)>0$ and $\rho_\delta:=t_\delta\wedge\tau^0$, where $\tau^0:=\tau^0_{t,x}:=\inf\{u\in[0,1-t]\, :\, X_{t+u}^{t,x}\leq 0\}$. Then, using the integral form of \eqref{MarkovX}, for an arbitrary $s\in[0,1-t]$ we have, $\P$-a.s.
	\begin{align}\label{eq:Y}
	Y^\eps_{t+s\wedge\rho_\delta}	&= -\int_0^{s\wedge\rho_\delta} \frac{X^{t,x}_{t+u}}{1-(t+u)}\mathrm{d} u+\int_0^{s\wedge\rho_\delta} \frac{X^{t-\eps,x}_{t-\eps+u}}{1-(t-\eps+u)}\mathrm{d} u\\
	&=-\int_0^{s\wedge\rho_\delta}\left( \frac{\eps X^{t,x}_{t+u}}{(1-(t-\eps+u))(1-(t+u))}+\frac{Y^{\eps}_{t+u}}{1-(t-\eps+u)}\right)\mathrm{d} u. \nonumber
	\end{align}

Let $[x]^+:=\max\{0,x\}$. Since $Y^\eps$ is a continuous process of bounded variation and $Y^\eps_0=0$, we have
\begin{align}
[Y^\eps_{t+s\wedge\rho_\delta}]^+=\int_0^{s\wedge\rho_\delta}\ind_{\{Y^\eps_{t+u}\ge 0\}}\ud Y^\eps_{t+u}\le 0
\end{align}	
where the final inequality follows from \eqref{eq:Y}, upon observing that $X^{t,x}_{t+u}\geq 0$ for all $u\leq\rho_\delta$.
Then, $Y^{\eps}_{t+s\wedge \rho_\delta}\leq 0$ for all $s\in[0,1-t]$. Furthermore, letting $\delta\to0$, we obtain by continuity of paths
	\begin{equation}\label{OrderOptStoppTimes}
X^{t-\eps,x}_{t-\eps+s\wedge \tau^0}\geq X^{t,x}_{t+s\wedge \tau^0}\ge 0,\quad\text{ for all } s\in[0,1-t], \: \P\text{-a.s.}
	\end{equation}
Hence, the process $X^{t,x}$ hits zero before the process $X^{t-\eps,x}$ does.
\vspace{+4pt}

{\em Step 2.} ($\partial_t v(t,x)\le 0$).	Fix $(t,x)\in\cC$ with $x>0$. Using the same notation as in step 1 above, let $\sigma:=\tau^*_{t,x}\wedge \tau^0_{t,x}$. By the (super)martingale property of the value function, noticing that $\tau^*$ is optimal in $v(t,x)$ and sub-optimal in $v(t-\eps,x)$ we have 
	\begin{align}\label{eq:vt1}
	&v(t,x)-v(t-\eps,x) \\
	&\leq \E\left[v(t+\sigma,X^{t,x}_{t+\sigma})-v(t-\eps+\sigma, X^{t-\eps,x}_{t-\eps+\sigma})\right]\notag \\
	&\leq\E\left[\ind_{\{\tau^*\leq\tau^0\}\cap\{\tau^*<1-t\}}\left(\exp(X^{t,x}_{t+\tau^*})-\exp(X^{t-\eps,x}_{t-\eps+\tau^*})\right)\right]\notag \\
	& \quad +\E\left[\ind_{\{\sigma=1-t\}}\left(\exp(X^{t,x}_{1})-\exp(X^{t-\eps,x}_{1-\eps})\right)\right]\notag \\
	& \quad +\E\left[\ind_{\{\tau^0<\tau^*\}\cap\{\tau^0<1-t\}}\left(v(t+\tau^0,0)-v(t-\eps+\tau^0,X_{t-\eps+\tau^0}^{t-\eps,x})\right)\right]. \notag
	\end{align}
Recalling \eqref{OrderOptStoppTimes}, on the event $\{\tau^*\leq\tau^0\}\cap\{\tau^*<1-t\}$ we have $X^{t-\eps,x}_{t-\eps+\tau^*}\geq X^{t,x}_{t+\tau^*}$ and on the event $\{\sigma=1-t\}$ we have that $X^{t-\eps,x}_{1-\eps}\geq X^{t,x}_{1}$. Moreover, $x\mapsto v(t,x)$ is non-decreasing (Proposition \ref{xVcontinuity}). Thus, combining these facts with \eqref{eq:vt1}, we obtain
\begin{align}\label{eq:vt2}
&v(t,x)-v(t-\eps,x) \\
&\leq\E\left[\ind_{\{\tau^0<\tau^*\}\cap\{\tau^0<1-t\}}\left(v(t+\tau^0,0)-v(t-\eps+\tau^0,0)\right)\right]\leq 0,\notag
\end{align}
where the final inequality uses \eqref{eq:dtv0} and the fact that $\tau^0<1-t$.
	
Finally, dividing both sides of \eqref{eq:vt2} by $\eps$ and letting $\eps\to 0$, we obtain $\partial_t v(t,x)\le 0$ as needed.
\end{proof}

It is well-known in optimal stopping theory that monotonicity of the boundary leads to its right-continuity (or left-continuity). In our case we have a simple corollary.
\begin{corollary}\label{cor:brc}
The boundary is right-continuous, whenever finite.
\end{corollary}
\begin{proof}
Let $t\in[0,1)$ be such that $b(t)<+\infty$. Consider a sequence $(t_n)_{n\in \N}$ such that $t_n\downarrow t$ as $n\to\infty$. By monotonicity of $b$ and \eqref{eq:bC}, we have that $b(t_n)<\infty$ and $(t_n,b(t_n))\in \cD$ for all $n\in \N$. Since $\cD$ is a closed set and $(t_n,b(t_n))\to (t,b(t+))$, then also $(t,b(t+))\in \cD$ (the right-limit $b(t+)$ exists by monotonicity). Hence, $b(t+)\geq b(t)$ (see \eqref{eq:bC}). However, by monotonicity $b(t)\geq b(t+)$, which leads to $b(t)=b(t+)$.
\end{proof}

We can now show that the optimal boundary is continuous and bounded on $[0,1]$.
\begin{proposition}\label{prop:bfinite}
The optimal boundary $t\mapsto b(t)$ is continuous on $[0,1]$ and we have
\begin{equation}\label{bBounded}
\sup_{t\in[0,1]} b(t)<+\infty.
\end{equation}
\end{proposition}

The proof of Proposition \ref{prop:bfinite} relies upon four lemmas. First we state and prove those lemmas and then we prove the proposition.

\begin{lemma}\label{lemma:nonemptystop}
	For any $t\in[0,1)$ we have $\cD\cap ([t,1)\times \R)\neq\varnothing$.
\end{lemma}
\begin{proof}
Suppose by contradiction that this is not true and there exists $t\in[0,1)$ such that $\cD\cap ([t,1)\times \R)=\varnothing$. Then $\tau^*_{t',x}=1-t'$, $\P$-a.s.~for all $(t',x)\in[t,1)\times\R$, which implies $v(t',x)=1$. This, however, leads to a contradiction since immediate stopping gives $v(t',x)\ge \e^x>1$ for $x>0$ and any $t'\in[t,1)$.
\end{proof}
Notice that the lemma above implies that for any $t_1\in[0,1)$ there exists $t_2\in(t_1,1)$ such that $b(t_2)<+\infty$. This fact will be used in the next lemma.

\begin{lemma}\label{lemma:bfiniteafter0}
The boundary satisfies $b(t)<+\infty$ for all $t\in(0,1]$.
\end{lemma}

\begin{proof}	
By contradiction, let us assume that there is $t\in(0,1)$ such that $b(t)=+\infty$. Then, thanks to Lemma \ref{lemma:nonemptystop} and Corollary \ref{cor:brc} we can find $t'\in(t,1)$ such that $0\leq b(t')=:b_0<+\infty$ and $(t,t')\times\R\subseteq\cC$. Let $\sigma_0:=\inf\{s\in [0,1-t)\,:\,X_{t+s}^{t,x}\leq b_0\}\wedge (t'-t)$, then recalling $\tau^*_{t,x}$ as in \eqref{eq:tau*}, we immediately see that $\P(\tau^*_{t,x}\ge \sigma_0)=1$. Using the martingale property of the value function (see \eqref{eq:vtilde}), we obtain 
\begin{align*}
\tilde{v}(t,x) &= \EE_{t,x}\left[\exp\left(\int_0^{\sigma_0}\left(\frac 1 2 - \frac {X_{t+s}}{1-(t+s)}\right)\mathrm{d} s\right)\tilde{v}(t+\sigma_0,X_{t+\sigma_0})\right]\\
	&= \EE_{t,x}\left[\ind_{\{\sigma_0< t'-t\}}\exp\left(\int_0^{\sigma_0}\left(\frac 1 2 - \frac{X_{t+s}}{1-(t+s)}\right)\mathrm{d} s\right)\tilde{v}(t+\sigma_0,b_0)\right] \\
	& \quad + \EE_{t,x}\left[\ind_{\{\sigma_0 = t'-t\}}\exp\left(\int_0^{t'-t}\left(\frac 1 2 - \frac{X_{t+s}}{1-(t+s)}\right)\mathrm{d} s\right)\cdot 1 \right],
	\end{align*}
where we have used continuity of paths and the fact that on $\{\sigma_0=t'-t\}$ it must be $X_{t'}\ge b(t')=b_0$, $\PP_{t,x}$-a.s.

Moreover, since $X_{t+s}^{t,x}\geq b_0$ for $s\leq \sigma_0$, we have
	\begin{align}\label{eq:vtilde0}
	\tilde{v}(t,x) 
	&	\leq \EE_{t,x}\left[\ind_{\{\sigma_0< t'-t\}}\exp\left(\int_0^{\sigma_0}\left(\frac 1 2 - \frac{b_0}{1-(t+s)}\right)\mathrm{d} s\right)\tilde{v}(t+\sigma_0,b_0)\right] \\
	&\quad+ \EE_{t,x}\left[\ind_{\{\sigma_0=t'-t \}}\exp\left(\int_0^{t'-t}\left(\frac 1 2 - \frac{X_{t+s}\vee b_0}{1-(t+s)}\right)\mathrm{d} s\right) \right]\notag\\
	&\leq \EE_{t,x}\left[\ind_{\{\sigma_0< t'-t\}}\left(\frac{1-(t+\sigma_0)}{1-t}\right)^{b_0}\e^{\sigma_0/2} \right]\cdot c_1 \notag\\
	&\quad + \EE_{t,x}\left[\ind_{\{\sigma_0= t'-t\}}\exp\left(\int_0^{t'-t}\left(\frac 1 2 - \frac{X_{t+s}\vee b_0}{1-(t+s)}\right)\mathrm{d} s\right) \right]\notag\\
	&\leq c_1\e^{1/2}\PP_{t,x}(\sigma_0< t'-t)+\EE\left[\exp\left(\int_0^{t'-t}\left(\frac 1 2 - \frac{X^{t,x}_{t+s}\vee b_0}{1-(t+s)}\right)\mathrm{d} s\right) \right],\notag
	\end{align}
	where in the second inequality we have used \eqref{eq:vb} and $\tilde{v}(t,x)=\e^{-x}v(t,x)$. Now, we let $x\to\infty$ and notice that 
\[
\PP_{t,x}(\sigma_0< t'-t)\le \PP\Big(\inf_{s\in[t,t']}X^{t,x}_s<b_0\Big)
\]
so that the first term on the right-hand side of \eqref{eq:vtilde0} goes to zero. Similarly, given that $\lim_{x\to\infty} X^{t,x}_{t+s}=+\infty$ for any $s\in[0,t'-t]$, the second term goes to zero as well by the reverse Fatou's lemma. Then, recalling that $\tilde v\ge 1$, we reach the contradiction 
	\[
	\limsup_{x\to+\infty}\tilde{v}(t,x)\leq 0.
	\]
It follows that $b(t)<+\infty$ for all $t\in(0,1]$ since by definition $b(1)=0$.
\end{proof}

\begin{lemma}\label{lemma:bfiniteat0}
	We have $b(0)<+\infty$.
\end{lemma}	

\begin{proof}
Consider an auxiliary problem where the Brownian bridge is pinned at time $1+h$, for some $h>0$, and the time horizon of the 
optimisation is $1+h$. That is, let us set
\begin{equation}
v^h(t,x):=\sup_{0\leq\tau\leq 1+h-t}\E_{t,x}\Big[\e^{\widetilde{X}_{t+\tau}}\Big],
\end{equation}
where $\widetilde{X}$ is a Brownian bridge \eqref{MarkBrownBridge} pinned at time $1+h$.
	
By the same argument as in Section \ref{sec:PF}, it follows that $\mathsf{Law}(\widetilde{X}^{t,x})=\mathsf{Law}(\widetilde{Z}^{t,x})$, where
\[
\widetilde{Z}^{t,x}_s=\frac{1+h-s}{1+h-t}x+\sqrt{\frac{1+h-s}{1+h-t}}W_{s-t}, \quad\text{for $s\in[t,1+h]$.}
\]
Thus,
\begin{equation}\label{MarkovianPb1h}
v^h(t,x)=\sup_{0\leq\tau\leq 1+h-t}\E_{t,x}\Big[\e^{\widetilde{Z}_{t+\tau}}\Big]
\end{equation}
and,	since $\mathsf{Law}(Z^{t,x}_s,\,s\in[t,1])=\mathsf{Law}(\widetilde{Z}^{t+h,x}_{s+h},\,s\in[t,1])$ (compare \eqref{MarkovianPb1} with \eqref{MarkovianPb1h}), we also have that
\begin{align}\label{eq:vvh}
v(t,x)=v^h(t+h,x), \quad \text{for all $(t,x)\in[0,1]\times\R$}.
\end{align}

By the same arguments as for the original problem, we obtain that there exists a non-decreasing, right-continuous optimal boundary $t\mapsto b^h(t)$ such that 
\[
\cC^h:=\{(t,x)\in[0,1+h]\times\R: \: v^h(t,x)>e^x\}=\{(t,x)\in[0,1+h)\times\R: \: x<b^h(t)\}.
\]
Moreover, since the gain function $\e^x$ does not depend on time, using \eqref{eq:vvh} we obtain 
\[
b(t)=b^h(t+h), \quad \text{for all $t\in[0,1]$.}
\]
In particular, $b(0)=b^h(h)$ and $b^h(h)<+\infty$ by applying the result in Lemma \ref{lemma:bfiniteafter0} to the auxiliary problem.   
\end{proof}

Using ideas as in \cite{de2015note}, we can also prove left-continuity of the optimal boundary.

\begin{lemma}\label{lemma:bleftcontinuous}
	The optimal boundary $t\mapsto b(t)$ is left-continuous.
\end{lemma}

\begin{proof}
We first prove that the boundary is left-continuous for all $t\in(0,1)$ and then that its left limit at $t=1$ is zero, that is $b(1-)=0=b(1)$.
	
	Suppose, by contradiction, that there exists $t_0\in(0,1)$ such that $b(t_0-)>b(t_0)$ and consider an interval $[x_1,x_2]\subset (b(t_0),b(t_0-))$. By monotonicity of $b$, we have that $[0,t_0)\times[x_1,x_2]\subset \cC$. Now, pick an arbitrary, non-negative $\varphi\in\cC_c^{\infty}([x_1,x_2])$. Since \eqref{freeb1} holds in $[0,t_0)\times[x_1,x_2]$, then for any $t<t_0$
	we have
	\begin{align}\label{eq:varphi}
	0 &= \int_{x_1}^{x_2} [\partial_t v(t,y)+\cL v(t,y)]\varphi(y)\mathrm{d} y \\
	&\leq \int_{x_1}^{x_2}\cL v(t,y) \varphi(y)\mathrm{d} y = \int_{x_1}^{x_2} v(t,y) (\cL^* \varphi)(t,y) \mathrm{d} y, \nonumber
	\end{align}
where for the inequality we have used $\partial_{t}v\le 0$ (see proof of Proposition \ref{DecreasingBarrier}) and in the final equality we have applied integration by parts and used the adjoint operator
	\[
	(\cL^* \varphi)(t,y):=\frac 1 2 \varphi''(y)+\frac{1}{1-t}\cdot\frac{\ud}{\ud y}(y\cdot \varphi(y)).
	\]
	
	Taking limits as $t\to t_0$ and using dominated convergence theorem, we obtain
	\begin{align}\label{LeftContContrad}
	0 & \leq \lim_{t\uparrow t_0} \int_{x_1}^{x_2} v(t,y) (\cL^* \varphi)(t,y) \mathrm{d} y=\int_{x_1}^{x_2} v(t_0,y)(\cL^* \varphi)(t_0,y)\mathrm{d} y  \\
	&= \int_{x_1}^{x_2} \e^y (\cL^* \varphi)(t_0,y) \mathrm{d} y  = \int_{x_1}^{x_2} \e^y \left(\frac{1}{2}-\frac{y}{1-t_0} \right) \varphi(y) \mathrm{d} y,\nonumber 
	\end{align}
	where we have used that $v(t_0,y)=\e^y$ and integration by parts in the final equality. 
	
	Finally, recalling that $x_1\geq b(t_0)>\frac{1-t_0}{2}$, then \eqref{LeftContContrad} leads to a contradiction because the right-hand side of the expression is strictly negative (also $\varphi$ is arbitrary).
	
	In order to prove that $b(1-)=b(1)=0$, we need a slight modification of the argument above. In particular, suppose by contradiction that $b(1-)>0$ and consider an interval $[x_1,x_2]\subset (0,b(1-))$. Then, replacing $\varphi$ in \eqref{eq:varphi} with $\tilde{\varphi}(t,x):=(1-t)\varphi(x)$, and using the same arguments with $t_0=1$ we reach a contradiction, i.e.
	\[
	0\leq \lim_{t\uparrow 1}\int_{x_1}^{x_2} v(t,y)(\cL \tilde{\varphi})(t,y)\mathrm{d} y=\int_{x_1}^{x_2} \e^y\frac{\mathrm{d}}{\mathrm{d}y}(y\cdot \varphi(y))\mathrm{d}y=-\int_{x_1}^{x_2} \e^y y \varphi(y)\mathrm{d}y<0.
	\]
\end{proof}

We are now able to prove Proposition \ref{prop:bfinite}.

\begin{proof}[Proof of Proposition \ref{prop:bfinite}]
The proof of \eqref{bBounded} follows immediately from Lemma \ref{lemma:bfiniteafter0} and Lemma \ref{lemma:bfiniteat0}. Right-continuity of the boundary follows from Corollary \ref{cor:brc} and \eqref{bBounded}, whereas left-continuity follows from Lemma \ref{lemma:bleftcontinuous}. Thus, the optimal boundary is bounded and continuous on $[0,1]$.
\end{proof}

\section{Regularity of the value function and integral equations}\label{sec:C1}

Thanks to monotonicity of the optimal boundary and to the law of iterated logarithm (combined with \eqref{MarkBrownBridge1}), it is easy to see that for any $(t,x)\in[0,1)\times\R$ it holds $\tau^*_{t,x}=\tau'_{t,x}$, $\P$-a.s., where 
\begin{align}\label{eq:tau2}
\tau'_{t,x}:=\inf\{s\in [0,1-t]: X^{t,x}_{t+s}>b(t+s)\},\quad\text{$\P$-a.s.}
\end{align}
That is, the first time the process reaches the optimal boundary it also goes strictly above it.
(A proof of this claim can be found, e.g., in Lemma 5.1 of \cite{de2017dividend}).

Moreover, combining \eqref{eq:tau2} with continuity of the optimal boundary, we deduce
\[
\tau^*_{t,x}=\inf\{s\in[0,1-t]:(t+s,X^{t,x}_{t+s})\in\text{int}(\cD)\},
\]
where $\text{int}(\cD)=\cD\setminus\partial\cC$ is the interior of the stopping set. In particular, since $\tau_{t,x}^*=0$, $\P$-a.s.~for any $(t,x)\in\partial\cC$, by its definition \eqref{Tau*}, this implies $\tau'_{t,x}=0$, $\P$-a.s.~as well for $(t,x)\in\partial\cC$. This means that the boundary $\partial\cC$ is regular for the interior of the stopping set in the sense of diffusions (see, e.g., \cite{de2018global}). 

It is therefore possible to prove (see, e.g., Corollary 6 in \cite{de2018global} and Proposition 5.2 in \cite{de2017dividend}) that for any $(t_0,x_0)\in\partial\cC$ (i.e., $x_0 =b(t_0)$) and any sequence $(t_n,x_n)_{n\ge 1}\subseteq\cC$ converging to $(t_0,x_0)$ as $n\to\infty$, we have
\begin{align}\label{eq:convtau}
\lim_{n\to\infty}\tau^*_{t_n,x_n}=\lim_{n\to\infty}\tau'_{t_n,x_n}=0,\quad\text{$\P$-a.s.}
\end{align}
Now we can use this property of the optimal stopping time and some related ideas from \cite{de2018global} to establish $C^1$ regularity of the value function.

First, we give a lemma concerning the spatial derivative of $v$.

\begin{lemma}\label{lem:vx}
	For all $(t,x)\in([0,1]\times\R)\setminus\partial\cC$ we have
	\begin{equation}\label{eq:vx}
	\partial_x v(t,x)=\E_{t,x}\Big[\frac{1-t-\tau^*}{1-t} \e^{Z_{t+\tau^*}}\Big].
	\end{equation}
	Hence, it also holds 
	\begin{align}\label{eq:vvx}
	\partial_x v(t,x)\le v(t,x),\qquad\text{for $(t,x)\in([0,1]\times\R)\setminus\partial\cC$}.
	\end{align}
\end{lemma}
\begin{proof}
	Recall that $v\in C^{1,2}(\cC)$ (see comment before \eqref{freeb1}). Moreover, $v(t,x)=\e^x$ on $\cD$ and $\partial_x v(t,x)=\e^x$ on $\cD\setminus\partial\cC$ 
	as needed in \eqref{eq:vx}. It remains to show that \eqref{eq:vx} holds for all $(t,x)\in \cC$. 
	
	Fix $(t,x)\in \cC$ and take $\varepsilon>0$. Recall problem formulation \eqref{MarkovianPb1} with the explicit expression for $Z$ (see \eqref{MarkBrownBridge1}) and recall that we use the notation $\tau^*:=\tau^*_{t,x}$ (as in \eqref{Tau*}). Since $\tau^*$ is admissible but sub-optimal for the problem with value $v(t,x+\eps)$, we have
	\begin{align*}
	v(t,x)-v(t,x+\eps) &\leq \E\left[\exp(Z^{t,x}_{t+\tau^*})-\exp(Z^{t,x+\varepsilon}_{t+\tau^*})\right] \\
	&= \E\left[\left(1-\exp\left(\frac{1-t-\tau^*}{1-t}\eps\right)\right)\exp(Z^{t,x}_{t+\tau^*})\right].
	\end{align*}
	Hence, by dominated convergence theorem and recalling that $v$ is differentiable at $(t,x)\in\cC$, we obtain 
	\begin{align}\label{eq:vx1}
	\partial_x v(t,x)=\lim_{\eps\to 0} \frac{v(t,x+\varepsilon)-v(t,x)}{\varepsilon}\geq \E\left[\frac{1-t-\tau^*}{1-t}\exp(Z^{t,x}_{t+\tau^*})\right].
	\end{align}
	By the same arguments, we also have that
	\[
	v(t,x)-v(t,x-\varepsilon)\leq \E\left[\left(1-\exp\left(-\frac{1-t-\tau^*}{1-t}\eps\right)\right)\exp(Z^{t,x}_{t+\tau^*})\right]
	\]
	which implies
	\begin{align}\label{eq:vx2}
	\partial_x v(t,x)=\lim_{\eps\to 0} \frac{v(t,x)-v(t,x-\varepsilon)}{\varepsilon}\leq \E\left[\frac{1-t-\tau^*}{1-t}\exp(Z^{t,x}_{t+\tau^*})\right].
	\end{align}
	Combining \eqref{eq:vx1} and \eqref{eq:vx2} we obtain \eqref{eq:vx}.
	
	Now, the inequality in \eqref{eq:vvx} follows easily by comparison of \eqref{eq:vx} and \eqref{MarkovianPb1}.
\end{proof}

\begin{theorem}\label{thm:C1}
We have $v\in C^1([0,1)\times\R)$. 
\end{theorem}
\begin{proof}
We know from \eqref{freeb1} that $\partial_x v$ and $\partial_t v$ exist and are continuous in $\cC$. Moreover, $v(t,x)=\e^x$ on $\cD$ implies $\partial_x v(t,x)=\e^x$ and $\partial_t v(t,x)=0$ for $(t,x)\in\cD\setminus\partial\cC$. Then, it remains to prove that $\partial_x v$ and $\partial_t v$ are continuous across the boundary $\partial\cC$. We do this in two steps below.
\vspace{+4pt}

\emph{Step 1}. (Continuity of $\partial_x v$). Fix $(t_0,x_0)\in\partial\cC$ with $t_0<1$ and recall \eqref{eq:vx}. Then, for any sequence  $(t_n,x_n)_{n\ge 1}\subseteq\cC$ converging to $(t_0,x_0)$ as $n\to\infty$, we can use dominated convergence theorem, continuity of paths and \eqref{eq:convtau} to obtain
\[
\lim_{n\to\infty}\partial_x v(t_n,x_n)=\E\left[\lim_{n\to\infty}\frac{1-t_n-\tau^*_{t_n,x_n}}{1-t_n}\exp(Z^{t_n,x_n}_{t_n+\tau^*_{t_n,x_n}})\right]= \e^x.
\]
\vspace{+4pt}

\emph{Step 2}. (Continuity of $\partial_t v$).	Let $(t,x)\in\cC$ and $0<\varepsilon<1-t$. Then, repeating arguments as those used in \eqref{eq:lip1} and recalling that $t\mapsto v(t,x)$ is non-increasing on $[0,1)$ (see proof of Proposition \ref{DecreasingBarrier}) we obtain 
\begin{align*}
0 &\geq v(t+\varepsilon,x)-v(t,x) \\
	&\geq \E\left[\e^{(1-\theta^*)x}\left(\e^{\sqrt{(1-\theta^*)(1-t-\varepsilon)}W_{\theta^*}}-\e^{\sqrt{(1-\theta^*)(1-t)}W_{\theta^*}}\right)\right]\\
	& \geq -\left(\sqrt{1-t}-\sqrt{1-t-\varepsilon}\right)\e^{|x|}\E\left[|W_{\theta^*}|\e^{|W_{\theta^*}|}\right],
	\end{align*}
where $\theta^*:=\theta^*_{t,x}$ is the optimal stopping time for $v(t,x)$ (see \eqref{MarkovianPb2}).
	
Dividing all terms above by $\eps$ and letting $\eps\to 0$, we find
\begin{align}\label{TimeDerivativeFormula}
0 \geq \partial_t v(t,x) \geq -\frac{1}{2\sqrt{1-t}}\e^{|x|}\E\left[|W_{\theta^*}|\e^{|W_{\theta^*}|}\right].
\end{align}
The inequalities in \eqref{TimeDerivativeFormula} hold if we replace $(t,x)$ by $(t_n,x_n)$ and $\theta^*$ by $\theta^*_n:=\theta^*_{t_n,x_n}$, where the sequence $(t_n,x_n)_{n\ge 1}\subseteq\cC$ converges to $(t_0,x_0)\in\partial\cC$ as $n\to\infty$.

Now we aim at letting $n\to\infty$.	Notice that \eqref{eq:convtau} and the definition of $\theta$ in \eqref{MarkovianPb2} imply 
\[
\lim_{n\to\infty}\theta^*_{t_n,x_n}=\lim_{n\to\infty}\frac{\tau^*_{t_n,x_n}}{1-t_n}=0, \quad \text{$\P$-a.s. }
\]
Thus, using dominated convergence theorem, we obtain 
\[
0\ge \lim_{n\to\infty} \partial_t v(t_n,x_n)\ge -\frac{1}{2\sqrt{1-t_0}}\e^{|x_0|}\E\left[\lim_{n\to\infty}|W_{\theta^*_n}|\e^{|W_{\theta^*_n}|}\right]=0.
\]  
\end{proof}
Theorem \ref{thm:C1} has a simple corollary which shows the regularity of $\partial_{xx} v$ across the boundary. In particular, $\partial_{xx}v$ is continuous but for a (possible) jump along the optimal boundary. 
\begin{corollary}\label{cor:C2}
The second derivative $\partial_{xx}v$ is continuous on $([0,1]\times\R)\setminus\partial\cC$. Moreover, for any $(t_0,x_0)\in\partial\cC$ with $t_0<1$ and any sequence  $(t_n,x_n)_{n\ge 1}\subseteq\cC$ converging to $(t_0,x_0)$ as $n\to\infty$, we have
\begin{align}\label{eq:vxx}
\lim_{n\to\infty}\partial_{xx}v(t_n,x_n)=\frac{2x_0}{1-t_0}\e^{x_0}\geq\e^{x_0}.
\end{align}
\end{corollary}
\begin{proof}
Since $v(t,x)=\e^x$ in $\cD$, then $\partial_{xx}v(t,x)=e^x$ in $\cD\setminus\partial\cC$ which is continuous. Moreover, $\partial_{xx}v \in C\big(\cC\big)$ and so $\partial_{xx}v$ is continuous on $([0,1]\times\R)\setminus\partial\cC$.

To show \eqref{eq:vxx}, it is sufficient to take limits in \eqref{freeb1}, that is
\begin{align}
\lim_{n\to\infty}\partial_{xx}v(t_n,x_n)=\lim_{n\to\infty}2\left(-\partial_tv(t_n,x_n)+\frac{x_n}{1-t_n}\partial_xv(t_n,x_n)\right)=\frac{2x_0}{1-t_0}\e^{x_0},
\end{align}
where we used Theorem \ref{thm:C1} to arrive at the final expression. The inequality in \eqref{eq:vxx} follows from the fact that $\cQ\subseteq\cC$ (see \eqref{eq:Q}).
\end{proof}

\subsection{Integral equation for the optimal boundary}\label{sec:II}

The regularity of the value function proved in the previous section allows us to derive an integral equation for the optimal boundary. This follows well-known steps (see, e.g., \cite{peskir2006optimal}) which we repeat briefly below.
\begin{theorem}
	For all $(t,x)\in[0,1)\times\R$, the value function has the following representation
	\begin{equation}\label{ValueFunctionRepr}
	v(t,x)=1+\E_{t,x}\left[\int_0^{1-t}\e^{X_{t+s}}\Big(\frac{X_{t+s}}{1-t-s}-\frac 1 2 \Big)\ind_{\{X_{t+s}>b(t+s)\}}\mathrm{d} s\right].
	\end{equation}
	Moreover, the optimal boundary $t\mapsto b(t)$ is the unique continuous solution of the following nonlinear integral equation, for all $t\in[0,1]$
	\begin{equation}\label{IntegralEq}
	\e^{b(t)}=1+\E_{t,b(t)}\left[\int_0^{1-t}\e^{X_{t+s}}\Big(\frac{X_{t+s}}{1-t-s}-\frac 1 2 \Big)\ind_{\{X_{t+s}>b(t+s)\}}\mathrm{d} s\right],
	\end{equation}
	with $b(1)=0$ and $b(t)\ge (1-t)/2$.
\end{theorem}

\begin{proof}
Thanks to Theorem \ref{thm:C1} and Corollary \ref{cor:C2}, we can find a mollifying sequence $(v_n)_{n\geq 0}\subseteq C^\infty([0,1)\times\R)$ for $v$ such that (see Section 7.2 in \cite{gilbarg2015elliptic})
	\begin{equation}\label{RegularFirstDerivative}
	(v_n,\partial_x v_n, \partial_t v_n)\to(v,\partial_x v, \partial_t v)
	\end{equation}
	as $n\to\infty$, uniformly on compact sets, and
	\begin{equation}\label{RegularSecondDerivative}
	\lim_{n\to\infty} \partial_{xx}v_n(t,x)=\partial_{xx}v(t,x), \quad \text{for all $(t,x)\notin\partial\cC$.}
	\end{equation}

We let $(K_m)_{m\geq 0}$ be a sequence of compact sets increasing to $[0,1-\varepsilon]\times\R$ and for $t<1$ we define
	$$\tau_m:=\inf\{s\geq 0: \: (t+s,X^{t,x}_{t+s})\notin K_m\}\wedge(1-t-\varepsilon).$$ 
By an application of It\^o's formula to $v_n$ and noticing that $\P(X^{t,x}_{t+s}=b(t+s))=0$ for $s\in[0,1-t)$, we obtain
	\begin{align}
	v_n(t,x) &=\E_{t,x}\bigg[v_n(t+\tau_m,X_{t+\tau_m})\\
	& \hspace{3.5em}-\int_0^{\tau_m}\Big(\partial_t v_n(t+s,X_{t+s})+ \cL v_n(t+s,X_{t+s})\Big)\ind_{\{X_{t+s}\neq b(t+s)\}}\mathrm{d} s\bigg]. \nonumber
	\end{align}
	Now, since $(t+s,X_{t+s})_{s\le\tau_m}$ lives in a compact, letting $n\to \infty$ and applying dominated convergence theorem, by \eqref{RegularFirstDerivative} and \eqref{RegularSecondDerivative} we obtain
	\begin{align*}
	v(t,x) &= \E_{t,x}\left[ \vphantom{\frac 1 2} v(t+\tau_m,X_{t+\tau_m})\right. \\
	&\hspace{3.5em} -\left.\int_0^{\tau_m}\Big(\partial_t v(t+s,X_{t+s})+ \cL v(t+s,X_{t+s})\Big)\ind_{\{X_{t+s}\neq b(t+s)\}}\mathrm{d} s \right] \\
	&=\E_{t,x}\left[v(t+\tau_m,X_{t+\tau_m})+ \int_0^{\tau_m}\e^{X_{t+s}}\Big(\frac{X_{t+s}}{1-t-s}-\frac{1}{2} \Big)\ind_{\{X_{t+s}> b(t+s)\}}\mathrm{d} s \right],
	\end{align*}
	where in the second equality we have used \eqref{freeb1} and the fact that $v(t,x)=\e^x$ in $\cD$.
	
Notice that $\tau_m\to 1-t-\varepsilon$ as $m\to \infty$ and the integrand on the right-hand side of the above expression is non-negative. Recalling \eqref{eq:vb} and letting $m\to\infty$, we can apply dominated convergence theorem and monotone convergence theorem (for the integral term) in order to obtain
\[
v(t,x)=\E_{t,x}\left[v(1-\varepsilon,X_{1-\varepsilon})+ \int_0^{1-t-\varepsilon}\e^{X_{t+s}}\Big(\frac{X_{t+s}}{1-t-s}-\frac{1}{2} \Big)\ind_{\{X_{t+s}> b(t+s)\}}\mathrm{d} s \right].
\]
By the same arguments, letting $\varepsilon\to 0$ we obtain \eqref{ValueFunctionRepr}, i.e.
	\begin{align*}
	v(t,x) &= \E_{t,x}\left[v(1-,X_{1-})+ \int_0^{1-t}\e^{X_{t+s}}\Big(\frac{X_{t+s}}{1-t-s}-\frac{1}{2} \Big)\ind_{\{X_{t+s}> b(t+s)\}}\mathrm{d} s \right]\\
	&=1+\E_{t,x}\left[\int_0^{1-t}\e^{X_{t+s}}\Big(\frac{X_{t+s}}{1-t-s}-\frac{1}{2} \Big)\ind_{\{X_{t+s}> b(t+s)\}}\mathrm{d} s \right],
	\end{align*}
where in the second line we have used that, for $t_n<1$,
\[
1 \leq \liminf_{(t_n,x_n)\to(1,0)}v(t_n,x_n)\leq\limsup_{(t_n,x_n)\to(1,0)}v(t_n,x_n)\leq\limsup_{(t_n,x_n)\to(1,0)}\e^{|x_n|}\E[\e^{|W_{\tau_n^*}|}]=1,
\]
which follows from problem formulation \eqref{MarkovianPb1} with $\tau_n^*:=\tau^*_{t_n,x_n}$.
	
Now the integral equation \eqref{IntegralEq} is obtained by setting $(t,x)=(t,b(t))$ in \eqref{ValueFunctionRepr}. Uniqueness of the solution to such equation follows a standard proof in four steps that was originally developed in \cite{peskir2005american}. The same proof has since been repeated in numerous examples, some of which are available in \cite{peskir2006optimal}. Therefore, here we only give a brief sketch of the key arguments of the proof.

Suppose there exists another continuous function $c:[0,1]\to\R_+$ with $c(1)=0$ and such that for all $t\in[0,1]$ it holds $c(t)\geq (1-t)/2$ and
\begin{align}\label{eq:int2}
\e^{c(t)} = 1+\E_{t,c(t)}\left[\int_0^{1-t}\e^{X_{t+s}}\Big(\frac{X_{t+s}}{1-t-s}-\frac 1 2 \Big)\ind_{\{X_{t+s}>c(t+s)\}}\mathrm{d} s\right].
\end{align}
Then, define the function
\[
v^c(t,x):=1+\E_{t,x}\left[\int_0^{1-t}\e^{X_{t+s}}\Big(\frac{X_{t+s}}{1-t-s}-\frac 1 2 \Big)\ind_{\{X_{t+s}>c(t+s)\}}\mathrm{d} s\right],
\]
i.e., the analogue of \eqref{ValueFunctionRepr} but replacing $b(\cdot)$ therein with $c(\cdot)$. Since $c(\cdot)$ is assumed continuous and the Brownian bridge admits a continuous transition density, it is not hard to show that $v^c$ is continuous on $[0,1)\times \R$. Moreover, it is clear that $v^c(1,x)=1$ for $x\in\R$ and, by \eqref{eq:int2}, $v^c(t,c(t))=\e^{c(t)}$ for $t\in[0,1]$. 

The main observation in the proof is that the process
\begin{align}\label{eq:martvc}
s\mapsto v^c(t+s,X_{t+s})+\int_0^{s}\e^{X_{t+u}}\Big(\frac{X_{t+u}}{1-t-u}-\frac 1 2 \Big)\ind_{\{X_{t+u}>c(t+u)\}}\mathrm{d} u
\end{align}
is a $\P_{t,x}$-martingale for any $(t,x)\in[0,1)\times \R$ and, moreover, it is a continuous martingale for $s\in[0,1-t)$. Using such martingale property and following \cite{peskir2005american} one obtains in order: (i) $v^c(t,x)=\e^x$ for all $x\geq c(t)$ with $t\in[0,1]$ and (ii) $v(t,x)\geq v^c(t,x)$ for all $(t,x)\in[0,1]\times\R$. Using (i) and (ii), continuity\footnote{It is shown in \cite{de2017integral} that continuity of the boundaries can be relaxed to right/left-continuity.} of $b(\cdot)$ and $c(\cdot)$ and, again, the martingale property of \eqref{eq:martvc} one also obtains: (iii) $c(t)\leq b(t)$ for all $t\in[0,1]$ and (iv) $c(t)\geq b(t)$ for all $t\in[0,1]$. Hence, $c(t)=b(t)$ for all $t\in[0,1]$.
\end{proof}

\section{Numerical results}\label{sec:numerics}

In order to numerically solve the nonlinear Volterra integral equation \eqref{IntegralEq}, we apply a Picard scheme that we learned from \cite{yerkin2018picard}.

First, notice that equation \eqref{IntegralEq} can be rewritten as
\[
\e^{b(t)}=1+\int_0^{1-t}\bigg(\int_{b(t+s)}^\infty \e^y\Big(\frac{y}{1-t-s}-\frac 1 2 \Big)p(t,b(t),t+s,y)\ud y \bigg)\mathrm{d} s,
\]
where $p(t,x,t+s,y):=\partial_y\P(X^{t,x}_{t+s}\le y)$ is the transition density of the Brownian bridge.

Let $\Pi:=\{0:=t_0<t_1<\ldots<t_n:=1\}$ be an equispaced partition of $[0,1]$ with mesh $h=1/n$. The algorithm is initialised by setting $b^{(0)}(t_j):=0$ for all $j=0,1,\ldots, n$. Now, let $b^{(k)}(t_j)$ denote, for $j=0,1,\ldots, n$, the values of the boundary obtained after the $k$-th iteration. Then, the values for the $(k+1)$-th iteration are computed, for all $j=0,\ldots, n$, as
\begin{align}\label{eq:num1}
\e^{b^{(k+1)}(t_j)}\!=\!1\!+\!\int_0^{1-t_j}\!\!\bigg(\!\int_{b^{(k)}(t_j+s)}^\infty \!\e^y\Big(\frac{y}{1\!-\!t_j\!-\!s}\!-\!\frac 1 2 \Big)p(t_j,b^{(k)}(t_j),t_j\!+\!s,y)\ud y \!\bigg)\mathrm{d} s.
\end{align}

In particular, the inner integral with respect to $\ud y$ can be computed explicitly. Indeed, noticing that 
\[
p(t,x,t\!+\!s,y)=\frac{1}{\sqrt{2\pi \alpha(t,s)}}\exp\left(-\frac{(y-\beta(x,t,s))^2}{2\alpha(t,s)}\right),
\]
with $\beta(x,t,s):=x(1-t-s)/(1-t)$ and $\alpha(t,s):=s(1-t-s)/(1-t)$, we can now substitute this expression inside the integral. Then, tedious but straightforward algebra allows to reduce the exponent of $\e^yp(t_j,b^{(k)}(t_j),t_j\!+\!s,y)$ to an exact square plus a term independent of $y$. Thus, properties of the Gaussian distribution give
\begin{align*}
&I(t_j,b^{(k)}(t_j),t_j\!+\!s,b^{(k)}(t_j+s))\\
&:=\int_{b^{(k)}(t_j+s)}^\infty \e^y\Big(\frac{y}{1-t_j-s}-\frac 1 2 \Big)p(t_j,b^{(k)}(t_j),t_j\!+\!s,y)\ud y\\
&=\e^{\gamma^{(k)}(t_j,s)}\left[\frac{\zeta(t_j,s)}{\sqrt{2\pi}}\e^{-\frac{\xi^{(k)}(t_j,s)^2}{2}}+\left(\eta^{(k)}(t_j,s)-\frac 1 2\right)\left(1-\Phi\big(\xi^{(k)}(t_j,s)\big)\right)\right],
\end{align*} 
where $\Phi$ is the cumulative density function of a standard normal distribution, and
\begin{align*}
&\gamma^{(k)}(t_j,s) \!:=\!\frac{(2b^{(k)}(t_j)\!+\!s)(1\!-\!s\!-\!t_j)}{2(1\!-\!t_j)}, \qquad \eta^{(k)}(t_j,s)\!:=\!\frac{b^{(k)}(t_j)\!+\!s}{1\!-\!t_j}, \\
&\zeta(t_j,s) \!:=\!\frac{s}{(1\!-\!s\!-\!t_j)(1\!-\!t_j)}, \qquad  \xi^{(k)}(t_j,s)\!:=\!\left(\frac{b^{(k)}(t_j\!+\!s)}{1\!-\!s\!-\!t_j}\!-\!\eta^{(k)}(t_j,s)\right)\frac{1}{\zeta(t_j,s)}.
\end{align*}

The integral with respect to the time variable (i.e., the one in $\ud s$) is computed by a standard quadrature method.  Hence, \eqref{eq:num1} reduces to
\begin{align*}
e^{b^{(k+1)}(t_j)}=1+h\sum_{m=0}^{n-1-j}I\big(t_j,b^{(k)}(t_j),t_j+mh+\tfrac{h}{2},b^{(k)}(t_j+mh+\tfrac{h}{2})\big),
\end{align*}
where each $b^{(k)}(t_j+mh+\tfrac{h}{2})$ is computed by interpolation and we use the convention $\sum_{m=0}^{-1}=0$ for $j=n$. Finally,
\begin{equation}\label{PicardSolution}
	b^{(k+1)}(t_j)=\log\left(1+h\sum_{m=0}^{n-1-j}I\big(t_j,b^{(k)}(t_j),t_j+mh+\tfrac{h}{2},b^{(k)}(t_j+mh+\tfrac{h}{2})\big)\right).
\end{equation}

The algorithm stops when the numerical error $e_k:=\max_{j=0,\ldots,n} |b^{(k)}(t_j)-b^{(k-1)}(t_j)|$ fulfills the tolerance condition $e_k<\eps$, for some $\eps>0$. A numerical approximation of the optimal boundary is presented in Figure \ref{fig-Optimal-boundary}.

While a rigorous proof of the convergence of the scheme seems difficult and falls outside the scope of this work, in Figure \ref{fig-Error} we show that the numerical error $e_k$ converges to zero as the number of iterations increases. Moreover, the convergence is monotone, which results in a good stability of the scheme.

Finally, in Figure \ref{fig-Value-function} we plot the value function as a surface in the $(t,x)$-plane using \eqref{ValueFunctionRepr}. It is interesting to observe that, as predicted by the theory, the value function exhibits a jump at $\{1\}\times(-\infty,0)$.

\begin{remark}\label{re:numT}
As noted in Remark \ref{rem:T}, we could consider a Brownian bridge with a generic pinning time $T>t$ and nothing would change in our analysis. However, it may be interesting to observe that as $T\to\infty$ the Brownian bridge converges (in law) to a Brownian motion $W$. Thus, we also expect that the stopping problem \eqref{MarkovianPb} converges to the problem of stopping the exponential of a Brownian motion over an infinite time horizon. Since $t\mapsto \exp(x+W_t)$ is a sub-martingale, the optimal stopping rule is to never stop. This heuristics is confirmed by Figure \ref{fig-different-T} where we observe numerically that the continuation set expands as $T$ increases and, in the limit as $T\to +\infty$, the stopping set disappears.
\end{remark}

\begin{figure}[t]
	\centering
	\includegraphics[scale=0.5]{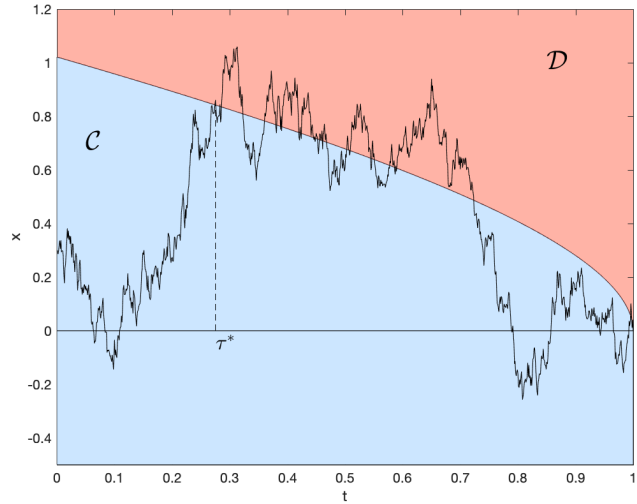}
	\caption{A sample path of a Brownian bridge $X$ starting at $X_0=0.3$ and pinned at $X_1=0$. The Brownian bridge hits the optimal boundary at $\tau^*\approx 0.3$. The boundary divides the state space into the continuation region $\cC$ (in blue) and the stopping region $\cD$ (in red). The tolerance of the algorithm is set to $\eps=10^{-6}$ and the equispaced time step is $h=10^{-3}$.}
	\label{fig-Optimal-boundary}
\end{figure}

\begin{figure}[t]
	\centering
	\includegraphics[scale=0.5]{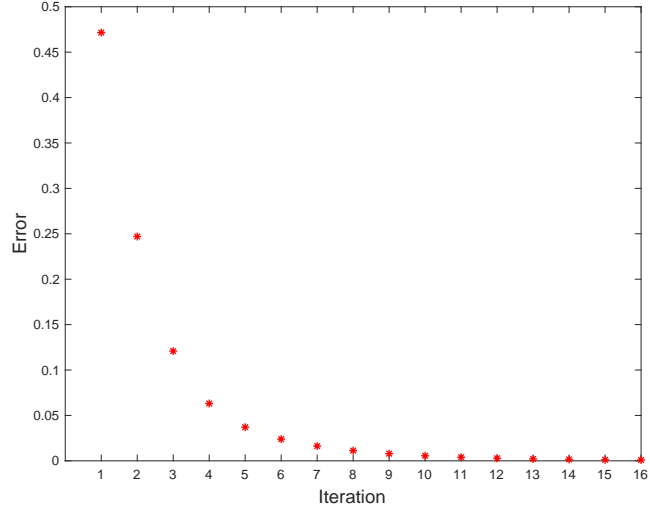}
	\caption{The trajectory of the error $e_k$ for the 16 iterations of the algorithm when $\eps=10^{-3}$.}
	\label{fig-Error}
\end{figure}

\begin{figure}[t]
	\centering
	\includegraphics[scale=0.5]{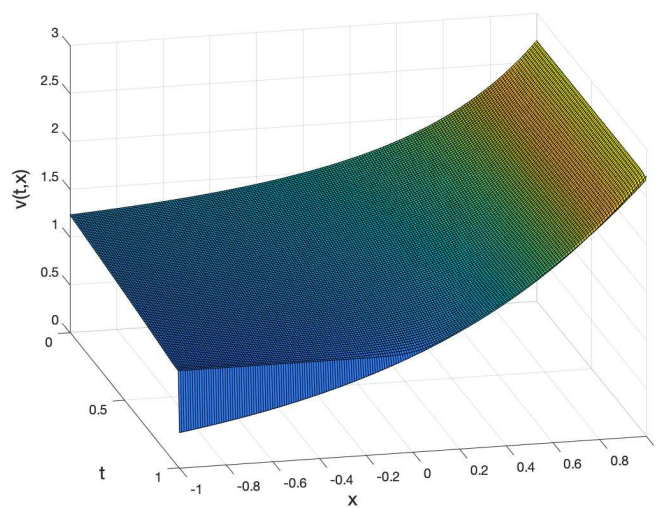}
	\caption{The value function surface $v(t,x)$ plotted on a grid of points $(t,x)\in[0,1]\times[-1,1]$ with discretization step $h=10^{-2}$.}
	\label{fig-Value-function}
\end{figure} 

\begin{figure}[t]
	\centering
	\includegraphics[scale=0.5]{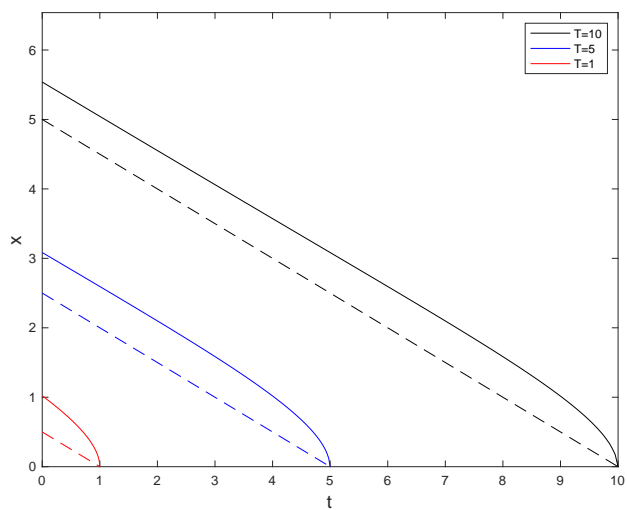}
	\caption{Boundary functions with starting point $t=0$ and increasing pinning times $T=1$ (continuous red line), $T=5$ (continuous blue line) and $T=10$ (continuous black line). Notice that every boundary lies above the corresponding line $x=\tfrac{1}{2}(T-t)$ (represented by the dashed lines in the same colours), which generalises the set $\cQ$ from \eqref{eq:Q}.}
	\label{fig-different-T}
\end{figure} 

\begin{remark}\label{rem:methods}
Traditionally, integral equations as in \eqref{IntegralEq} are solved by discretisation of the integral with respect to time and a backward procedure, starting from the terminal time (see, e.g., \cite[Ch.~VII, Sec.~27, pp.~432-433]{peskir2006optimal} for details in the case of the Asian option or \cite[Ch.~VIII, Sec.~30, p.~475]{peskir2006optimal} for another example; this method has been developed in the seminal paper \cite{kim1990analytic} and later extended). To the best of our knowledge, a rigorous proof of convergence for this `traditional' numerical scheme is not available. 

At each time step, the scheme must find the root of a highly non-linear algebraic equation, making the procedure slower than the Picard scheme that we implement, which requires no root finding (see Figure \ref{fig-Picard-Traditional} for a comparison).%For $h=10^{-3}$ and $\varepsilon=10^{-5}$, we typically observe that the Picard scheme stops after about 4 seconds whereas the `traditional' scheme stops after about 25 seconds.

Another possibility is to use finite-difference methods to solve directly the free boundary problem in \eqref{freeb1}--\eqref{freeb2}. The finite difference method, however, requires discretisation of both time and space (whereas we only discretise time) which leads to discretisation errors both in time and space and generally to a slower convergence. Moreover, in our case the coefficient associated to the first order partial derivative $\partial_x v$ is discontinuous at $t=1$, which causes additional difficulties.
\end{remark}

\begin{figure}[t]
	\centering
	\includegraphics[scale=0.5]{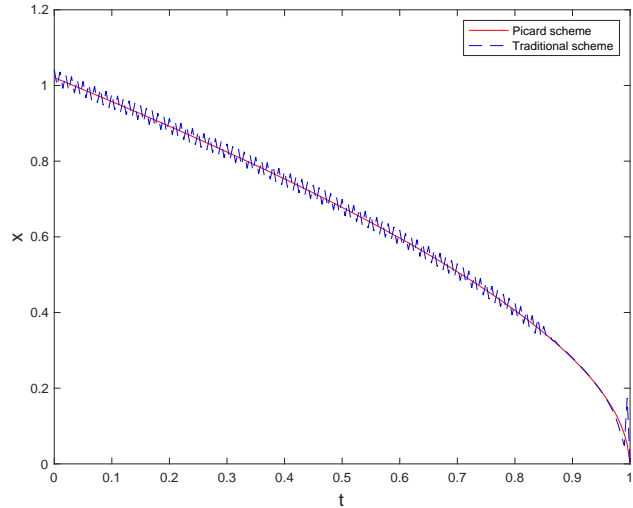}
	\caption{The solution of the optimal boundary found via Picard scheme (continuous line) and via traditional method (dashed line). The time step is $h=5\cdot 10^{-3}$ and the tolerance is $\varepsilon=10^{-5}$. The Picard scheme stops after 0.1 seconds and 36 iterations, the traditional method stops after 2.9 seconds.}
	\label{fig-Picard-Traditional}
\end{figure} 

\begin{figure}[t]
	\centering
	\includegraphics[scale=0.7]{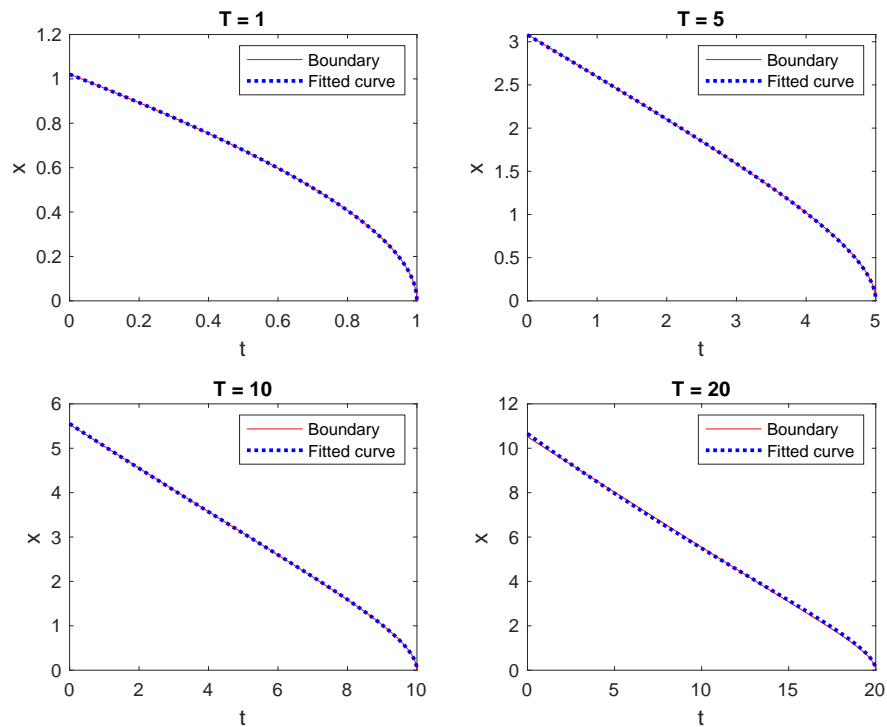}
	\caption{Boundary functions (in red) and the corresponding fitted curves (in blue) of the form $b_T(t)=A_T(1-\exp(B_T\sqrt{T-t}))$ for the pinning times $T=1,5,10,20$. The values of the parameters are, respectively, $A_1=-2.09$, $B_1=0.4$; $A_5=-1.85$, $B_5=0.43$; $A_{10}=-1.86$, $B_{10}=0.44$ and $A_{20}=-2.27$, $B_{20}=0.39$.}
	\label{fig-boundary-fits}
\end{figure} 

\begin{remark}
After we solved numerically equation \eqref{IntegralEq} and we were able to plot the boundary function as in Figure \ref{fig-Optimal-boundary}, we began to investigate whether a suitable fit for the boundary could be found. It is known that in the stopping problem with linear payoff $\E[X_\tau]$ and pinning time $T=1$, the optimal boundary can be found explicitly and it takes the form $b(t)=B\sqrt{1-t}$ (see \cite{shepp1969}). Motivated by this result we consider candidate boundaries of the form $b_T(t)=A_T(1-\exp(B_T\sqrt{T-t}))$, where $T$ is the pinning time of the Brownian bridge, and $A_T$ and $B_T$ are parameters to be determined.

Using the `curve fitting toolbox' in Matlab to fit our ansatz to the boundaries obtained from the integral equations, we obtain an excellent fit for $T=1,5,10$. The quality of the fit slightly deteriorates for larger $T$'s (e.g., for $T=20$). Results are illustrated in Figure \ref{fig-boundary-fits}. While these tests suggest that our problem might be amenable to an explicit solution, the question is more complex than in the linear payoff case (and its extensions in \cite{ekstrom2009optimal}) and remains open. The key difficulties are that (i) we must determine two parameters rather than one and (ii) we do not have a good guess for the value function that would allow us to transform the free-boundary problem \eqref{freeb1}-\eqref{freeb2} into a solvable ordinary differential equation. Indeed, in the linear case one uses the identity in law
\[
\mathsf{Law}\big(X^{t,x}_s\,,\, s\in[t,1]\big)=\mathsf{Law}\big(\widetilde{Z}^x_s\,,\,s\in[0,\infty)\big),
\]
with $\widetilde{Z}^{x}_s:=(x+\sqrt{1-t}W_{s})/(1+s)$, and obtains
\[
U(t,x):=\sup_{0\le \tau\le 1-t}\E[X^{t,x}_{t+\tau}]=\sup_{\tau\ge 0}\E\left[\frac{x+\sqrt{1\!-\!t}\,W_{\tau}}{1+\tau}\right]=\sqrt{1\!-\!t}\,U\left(0,\tfrac{x}{\sqrt{1-t}}\right).
\]
With the exponential payoff, the same identity does not provide any useful insight.

A different approach based on finding parameters for which the guessed boundary $b_T(t)$ solves the integral equation \eqref{IntegralEq} seems even harder.
\end{remark}

\vspace{+5pt}

\noindent{\bf Acknowledgment}: T.~De Angelis gratefully acknowledges support via EPSRC grant EP/R021201/1, ``\emph{A probabilistic toolkit to study regularity of free boundaries in stochastic optimal control}''. A.~Milazzo gratefully acknowledges support from Imperial College's Doctoral Training Centre in Stochastic Analysis and Mathematical Finance. Parts of this work were carried out while A.~Milazzo was visiting the School of Mathematics at the University of Leeds. Both authors thank the hospitality from the University of Leeds. Finally, we thank an anonymous referee whose useful suggestions improved the quality of the paper and in particular Section \ref{sec:numerics}.

\bibliography{bibfile}{}
\bibliographystyle{abbrv} 

\end{document}